\documentclass[reqno]{amsart}
\usepackage{hyperref}

\begin{document}

\title[\hfilneg  \hfil Some problems of summability of spectral expansions ]
{Some problems of summability of spectral expansions connected
with Laplace operator on sphere.}
\author[ Abdumalik A. Rakhimov, \hfil \hfilneg]
{Abdumalik A. Rakhimov}

\address{Abdumalik A. Rakhimov \newline
Department of Mathematical Physics, National University of
Uzbekistan, Tashkent, Republic of Uzbekistan}
\email{a\_rakhimov2008@yahoo.com}

\date{}
\thanks{}
\thanks{} 

\begin{abstract}
 Solution of some boundary value problems and initial problems in unique
 ball leads to the convergence and sumability problems of Fourier series of
 given function by eigenfunctions of Laplace operator on a sphere - spherical harmonics. Such a series are called as Fourier-Laplace
 series on sphere. There are a number of works devoted investigation of these expansions in different topologies and for the functions
 from the various functional spaces.
In the present work we consider only localization problems in both
usual and generalized (almost everywhere localization) senses  in
the classes of summable functions. We use Chezaro means of the
partial sums for study of summability problems. In order to prove
the main theorems we obtaine estimations for so called maximal
operator estimating it by Hardy-Littlwood's maximal function.
Significance of this function is that it majors of many important
operators of mathematical physics. For instance,
Hardy-Littlewood's maximal function majors Poisson's integral in
the space.

\end{abstract}

\maketitle \numberwithin{equation}{section}
\newtheorem{theorem}{Theorem}[section]
\newtheorem{lemma}[theorem]{Lemma}
\newtheorem{example}[theorem]{Example}
\newtheorem{definition}[theorem]{Definition}
\newtheorem{corollary}[theorem]{Corollary}

\section{Introduction}
\vspace{0.2 cm}

Denote by  $B^{N+1}$  a  unique ball in $R^{N+1}$ , surface of
this ball denote by  $S^N$  :
$$ S^N = \big\{  x=(x_1, x_2, ...... ,x_{N+1})\in R^{N+1}: \sum_{n=1}^{N+1} x_n^2 = 1 \big\} $$

Let   $x$     and   $y$    arbitrary points in  $S^N$   . By
$\gamma = \gamma(x, y) $ denote spherical distance between these
two points. In fact $\gamma$ is an angle between vectors  $x$ and
$y$ . It is clear that $\gamma \leq \pi$. By $B(x, r)$ denote a
ball on a sphere $S^N$ , with radius $r$ and with the center at a
point $x$ :
$$ B(x, r) = \big\{ y \in S^{N}: \gamma(x, y) \leq r \big\} $$

Let    $\Delta_s$   be Laplace-Beltrami operator on $S^N$. We have
following way to calculate operator  $\Delta_s$  , using Laplace's
operator $\Delta$ in  $R^{N+1}$    (see for instance in
\cite{shu1}.): let $f(x)$ a function determined on $S^N$    ;
extend it to $R^{N+1}$, by putting $\hat{f}(x) =
f\big{(}\frac{x}{|x|} \big{)} $, $x \in R^{N+1}$. Then $\Delta_s f
= \Delta \hat{f}\big{|}_{S^N}$. Another way of determination of
$\Delta_s$ is to represent Laplace operator $\Delta$ in $R^{N+1}$
by spherical coordinates. In this case it would be easy to
"separate" operator $ \Delta_s $  by separation angled
coordinates:
 $$ \Delta = \frac{\partial^2}{\partial{r^2}} + \frac{N}{r}
 \frac{\partial}{\partial{r}} + \frac{1}{r^2}\Delta_s ,$$
where operator $ \Delta_s $ can be written in spherical
coordinates $(\xi_1, \xi_2,...., \xi_{N-1}, \zeta)$ as:
\vspace{4mm}
$$ \Delta_s =
\frac{1}{\sin^{N-1}\xi_1}\frac{\partial}{\partial\xi_1}\Big{(}\sin^{N-1}{\xi_1}\frac{\partial}{\partial\xi_1}\Big{)}
+
\frac{1}{\sin^2\xi_1\sin^{N-2}\xi_2}\frac{\partial}{\partial\xi_2}\Big{(}\sin^{N-2}{\xi_2}\frac{\partial}{\partial\xi_2}\Big{)}
+ \ . \ . \ . \ . + $$

$$ + \frac{1}{\sin^2\xi_1 \sin^2\xi_2 \ . \
. \ . \sin^2\xi_{N-1}} \frac{\partial^2}{\partial\zeta^2} .$$

\vspace{4mm}

Operator  $ - \Delta_s$     as a formal differential operator with
domain of definition $C^{\infty}\big{(}S^N\big{)}$
 is a symmetric, non negative and its
closure $\overline{- \Delta_s}$ is a selfadjoint operator in
$L_2\big{(}S^N\big{)}.$   Eigenfunctions $Y^k$  of the operator  $
- \Delta_s$ , are called spherical harmonics. Spherical harmonics
of a degree $k$ and $\ell$ , $k \neq \ell$ \ \ are orthogonal
 . Corresponding eigenvalues are  $\lambda_k = k(k+N-1),$    where $k = 0, 1, 2, .
. .  .$ , and with frequency $a_k$ equal to the dimension of the
space of homogeneous harmonic polynomials of a degree $k$: \quad
$a_k = N_k - N_{k-2},$
     where   $N_k = \frac{(N + k)!}{N! k!}$    .
    That is why for each   $k$   there are   $a_k$    number of spherical
    harmonics
     $\big\{ Y^k_j \big\} \Big{|}^{a_k}_{j=1}$
corresponding to eigenvalue $\lambda_k$  . A family of functions
$\big\{ Y^k_j \big\} \Big{|}^{a_k}_{j=1}$ is an orthonormal  basis
in the space of spherical harmonics of a degree $k$ which we
denote by $\aleph_k$.

Note that an arbitrary function $f \in L_2\big{(}S^N\big{)}$ can
be represented in a unique way as Fourier series by spherical
harmonics $\big\{ Y^k_j \big\} \Big{|}^{a_k}_{j=1}$. Such a series
is called Fourier-Laplace series on sphere:

\begin{gather}\label{rce1}
f(x) = \sum_{k=0}^{\infty}  \sum_{j=1}^{a_k} f_{k, j }  Y_j^k(x),
 \end{gather}
 where  $f_{k,j}=\int_{S^N} f(y)Y_j^k(y) d\sigma(y)$ ,     and  equality (1.1)
 should be understanding in sense of $L_2\big{(}S^N\big{)}$ .

    Let denote by  $S_nf(x)$    a partial sum of series (1.1). It is clear that in   $S_nf(x)$ by
           changing order of integration and summation one can easily rewrite it as:
  $$S_nf(x) =\int_{S^N} f(y)\Theta(x, y, n) d\sigma(y),$$
where a function  $ \Theta(x, y, n) $    is a spectral function
(see in \cite{Ali} ) of a selfadjoint operator   $\overline{-
\Delta} $ and has a form:
\begin{gather}\label{rce1}
\Theta(x, y, n) = \sum_{k=0}^{n}  \sum_{j=1}^{a_k} Y_{j}^{k}(x)
Y_j^k(y),
 \end{gather}
 and \quad $S_nf(x)$ \quad is called a spectral expansion of an
 element \ $f$ \ correspondin to the operator $\overline{-
\Delta} $ (see in \cite{Ali} ).

\vspace{0.5 cm}
\section{Estimation of maximal operator and almost everywhere
convergence. } \vspace{0.2 cm}

 Determine Chezaro means of order    $\alpha $  of partial sums
of series   (1.1)  by equality
\begin{gather}\label{rce2}
S_n^{\alpha} f(x) = \frac{1}{A_{n}^{\alpha}} \sum_{k=0}^{n}  A_{n
- k}^{\alpha} \sum_{j=1}^{a_k} f_{k, j }  Y_j^k(x),
\end{gather}
where $A_{n}^{\alpha} = \frac{\Gamma(\alpha + m +
1)}{\Gamma(\alpha + 1)m!}.$
\begin{definition} \label{def2.1} \rm
Series  (1.1)  is sumable to $f(x)$   by Chezaro means of order
$\alpha$ if it is true that
\begin{gather}\label{rce2}
\lim_{n\rightarrow \infty}S_n^{\alpha} f(x) = f(x)
\end{gather}
\end{definition}
In the this definition equality  (2.2)  can be understood in any
sense (topology). Here in the present article we will consider it
in sense of almost every where convergence.

Note that Chezaro means of zero order is coincides with a partial
sum \quad $S_nf(x)$ \quad  and it is clear that \quad
$S^{\alpha}_nf(x)$ \quad is also can be represented as an integral
operator with a kernel which is Chezaro means of the spectral
function (1.2)
\begin{gather}\label{rce2}
\Theta^{\alpha} (x, y, n) = \frac{1}{A_{n}^{\alpha}}
\sum_{k=0}^{n} A_{n - k}^{\alpha} \sum_{j=1}^{a_k} Y_{j}^{k}(x)
Y_j^k(y),
\end{gather}
Thus formula   (2.1)  can be written as
\begin{gather}\label{rce2}
S_n^{\alpha}f(x) =\int_{S^N} f(y)\Theta^{\alpha}(x, y, n)
d\sigma(y).
\end{gather}
By $P_k^{\nu}(t)$   denote Gegenbaur's polinomials  (when $\nu =
\frac{1}{2}$ \quad Legender's polinomials) \cite{Kacz}, \cite{Sze}
. Let's put $\nu = \frac{N - 1}{2}$   .  A function $
\Theta^{\alpha} (x, y, n) $ can be represented in a form (see for
instance in \cite{kog1} ) :
\begin{gather}\label{rce2}
\Theta^{\alpha} (x, y, n) = \frac{\Gamma(n + 1)}{\Gamma(n + \alpha
+ 1)} \sum_{k = 0}^{n} \frac{\Gamma(n - k + \alpha + 1)}{\Gamma(n
- k + 1)}  \big{(} k + \nu \big{)}^{\nu} P_k^{\nu} (\cos\gamma).
\end{gather}

 In investigations of the convergence problems it is important to
obtain estimations of so called maximal operator
\begin{gather}\label{rce2}
S_{*}^{\alpha}f(x) = \sup_{n > 1} | S_n f(x) |.
\end{gather}
For any locally sumable on   $S^N$  function  $f(x)$   by
$f^{*}(x)$ denote Hardy-Littlwood's maximal function determined as
\begin{gather}\label{rce2}
f_{*}(x) = \sup_{r > 0} \frac{1}{mes B(x, r)}\int_{B(x, r)} | f(y)
| d\sigma(y).
\end{gather}
 In this section we prove that this function majors
operator determined by  (2.6). First we note here some well known
properties of Hardy-Littlwood's maximal function.

There exists a constant $c=c(N)$, depending only on dimension of
sphere such that for all $f \in L_1 \big{(} S^N \big{)}$ and
$F_{\mu} = \{x \in S^N: f^{*}(x) > \mu > 0 \}$ following
estimation is valid
\begin{gather}\label{rce2}
mes F_{\mu} \leq \frac{c \parallel f \parallel_1}{\mu}
\end{gather}
where  $\parallel f \parallel_1$    denotes a norm of \ $f$ \ in
$L_1 \big{(} S^N \big{)}$.

Moreover, there is a constant $b = b(p, N)$, that depends only on
dimension $N$ and on an index of summability $p
> 1$ such that
\begin{gather}\label{rce2}
\|f^{*}\|_{p} \leq b \|f\|_p.
\end{gather}
for all  $f \in L_p \big{(} S^N \big{)} $  .
    Estimation (2.8) is well known as Hardy-Littlwood's function
    has weak type of (1,1) (means it is weak bounded from   $L_1 \big{(} S^N
    \big{)}$
     to  $ L_1 \big{(} S^N \big{)} $) and (2.9)  says that it has strong (p, p) type when \ $p > 1$.

 Let's       denote by $\overline{x}$ a point that diametrically
opposite to a point $x$ of the sphere:  $\gamma(x, \overline{x}) =
\pi$   . In the present section we will prove following estimation
for maximal operator \ (2.6):
\begin{theorem} \label{Th.sol}
Let   $f(x)$    a summable  on a sphere function and let  $\alpha
> \frac{N - 1}{2}$ . Then there is a constant   $c_{\alpha}$
not depending on $f$ such that
 \begin{gather}\label{podmex3}
 S_{*}^{\alpha} f(x) \leq c_{\alpha} \big{(} f^{*}(x) + f^{*}(\overline{x})\big{)},
\end{gather}
where   $f^{*}$   is  Hardy-Littlewood's  maximal function.
\label{podmex1}
\end{theorem}
\begin{proof}
For the kernel   (2.5) following estimations are valid
\cite{kog1}:

if $\alpha > -1 $  and $|\frac{\pi}{2} - \gamma| \leq \frac{n}{n +
1} \frac{\pi}{2}$ , then for $n \rightarrow \infty$

$$ \Theta^{\alpha}(x, y, n) = \frac{\Gamma(\alpha + 1)}{\Gamma(n +
\alpha + 1)} \frac{\Gamma \Big{(} n + \frac{N +
1}{2}\Big{)}}{\Gamma \Big{(}\frac{N + 1}{2} \Big{)}}\times
\frac{\sin\Big{(} \Big{(} n + \frac{N + 1}{2}\Big{)}\gamma -
 \Big{(} \frac{N - 1}{2} + \frac{\alpha}{2}\Big{)}\frac{\pi}{2}
 \Big{)}}{\Big{(}2 \sin \gamma \Big{)}^{\frac{N - 1}{2}} \Big{(}2 \sin \frac{\gamma}{2} \Big{)}^{1 + \alpha}}
 + $$

\begin{gather}\label{rce2}
 + \frac{O \Big{(}n^{\frac{N - 1}{2} - \alpha - 1} \Big{)}}
{\Big{(} \sin \gamma \Big{)}^{\frac{N + 1}{2}} \Big{(} \sin
\frac{\gamma}{2} \Big{)}^{1 + \alpha}}
  + \frac{O \Big{(}\frac{1}{n} \Big{)}
 }{\Big{(} \sin \frac{\gamma}{2} \Big{)}^{1 + N}}
\end{gather}
if $\alpha > -1 $  and $0 < \gamma_0 \leq \gamma \leq \pi$ , then
for $n > 1$
\begin{gather}\label{rce2}
|\Theta^{\alpha}(x, y, n)| \leq c \ n^{N - 1 - \alpha}
 \end{gather}
if $\alpha > -1 $  and $0 \leq \gamma \leq \pi$ , then for $n
> 1$
\begin{gather}\label{rce2}
|\Theta^{\alpha}(x, y, n)| \leq c \ n^{N}
 \end{gather}

We will use these asymptotic formulas for estimation of Chezaro
means \ (2.4) . For that we transform integral in the right side
of (2.4) to the following form:

$$\int_{\gamma \leq \frac{1}{n}} \ + \ \int_{\frac{1}{n} \leq \gamma \leq \frac{\pi}{2}} \ + \
\int_{\frac{\pi}{2} \leq \gamma \leq \pi - \frac{1}{n}} \ + \int_{
\frac{1}{n} \leq \gamma \leq \pi }  $$ Then using correspondingly
formulas (2.11)-(2.13), obtain

$$ | S_n^{\alpha} f(x) | \ \leq \ c_1 \ n^N \ \int_{\gamma \leq
\frac{1}{n}} \ | f(y) | d \sigma(y) \ + \ c_2 \ n^{\frac{N - 1}{2}
- \alpha} \ \int_{\frac{1}{n} \leq \gamma \leq \frac{\pi}{2}} \
\frac{| f(y) |} {\Big{(}\sin \gamma \Big{)}^{ \frac{N + 1}{2}+
\alpha}} d \sigma(y) \ + $$ $$+ \ c_3 \ n^{\frac{N - 3}{2} -
\alpha} \ \int_{\frac{1}{n} \leq \gamma \leq \frac{\pi}{2}} \
  \frac{|
f(y) |} {\Big{(}\sin \gamma \Big{)}^{ \frac{N + 3}{2} + \alpha}} d
\sigma(y) +   \  c_4 \ n^{-1}\ \int_{\frac{1}{n} \leq \gamma \leq
\frac{\pi}{2}} \
  \frac{|
f(y) |} {\Big{(}\sin \gamma \Big{)}^{ N + 1}} d \sigma(y) + $$ $$
+ \ \ c_5 \ n^{\frac{N - 1}{2} - \alpha} \ \int_{\frac{\pi}{2}
\leq \gamma \leq \pi - \frac{1}{n}} \  \frac{| f(y) |}
{\Big{(}\sin \gamma \Big{)}^{ \frac{N - 1}{2}+ \alpha}} d
\sigma(y) \ +  \ c_6 \ n^{\frac{N - 3}{2} - \alpha} \
\int_{\frac{\pi}{2} \leq \gamma \leq \pi - \frac{1}{n}} \  \frac{|
f(y) |} {\Big{(}\sin \gamma \Big{)}^{ \frac{N + 1}{2}+ \alpha}} d
\sigma(y) \ + $$ \begin{gather}\label{rce2} + \ c_7 \ n^{-1} \
\int_{\frac{\pi}{2} \leq \gamma \leq \pi - \frac{1}{n}} \ | f(y) |
d \sigma(y) \ + \ c_8 \ n^{N} \  \int_{ \frac{1}{n} \leq \gamma
\leq \pi } \ | f(y) | d \sigma(y)
 \end{gather}
 Further we continue to estimate through Hardy-Littlwood's maximal function.

$$ | S_n^{\alpha} f(x) | \ \leq \ c_0 \ \Big{\{} \ n^N \ \int_{\gamma \leq
\frac{1}{n}} \  | f(y) | d \sigma(y) \ + $$

 $$ + \ n^{\frac{N -1}{2} - \alpha} \ \int_{\frac{1}{n}}^{\frac{\pi}{2}} \
 \frac{1}{\Big{(} \sin \gamma \Big{)}^{ \frac{N + 3}{2} +
\alpha}} \ d \ \int_{ \gamma < r } \ | f(y) | \ d \sigma(y) \ +
$$
$$ + \ n^{\frac{N -3}{2} - \alpha} \ \int_{\frac{1}{n}}^{\frac{\pi}{2}} \
 \frac{1}{\Big{(} \sin \gamma \Big{)}^{ \frac{N + 3}{2} +
\alpha}} \ d \ \int_{ \gamma < r } \ | f(y) | \ d \sigma(y) \ +
$$
$$ + \ \frac{1}{n} \ \int_{\frac{1}{n}}^{\frac{\pi}{2}} \
 \frac{1}{\Big{(} \sin \gamma \Big{)}^{ N + 1}} \ d \ \int_{ \gamma < r } \ | f(y) | \ d \sigma(y) \ +
$$
 $$ + \ n^{\frac{N -1}{2} - \alpha} \ \int_{\frac{1}{n}}^{\frac{\pi}{2}} \
 \frac{1}{\Big{(} \sin \gamma \Big{)}^{ \frac{N - 1}{2} }} \ d \ \int_{ \gamma < r } \ | f(y) | \ d \sigma(y)
\ +
$$
$$ + \ n^{\frac{N -3}{2} - \alpha} \ \int_{\frac{1}{n}}^{\frac{\pi}{2}} \
 \frac{1}{\Big{(} \sin \gamma \Big{)}^{ \frac{N + 1}{2}}} \ d \ \int_{ \gamma < r } \ | f(y) | \ d \sigma(y) \ +
$$
\begin{gather}\label{rce2} + \ \frac{1}{n} \
\int_{\frac{1}{n}}^{\frac{\pi}{2}} \
 \ d \ \int_{ \gamma < r } \ | f(y) | \ d \sigma(y) \ \Big{\}}
\end{gather}

By replacing with maximal function we obtain
$$ | S_n^{\alpha} f(x) | \ \leq \ c_0 \ \Big{\{} \ \Big{(} 1 + \ n^{\frac{N -1}{2} - \alpha} \ \int_{\frac{1}{n}}^{\frac{\pi}{2}} \
 r^{ \frac{N - 3}{2} - \alpha} dr \ + \ \frac{1}{n} \ \int_{\frac{1}{n}}^{\frac{\pi}{2}} \
 \frac{dr}{r^2}\Big{)} \ f^{*}(x) \ + $$
$$
 + \ \Big{(} 1 + \ n^{\frac{N -1}{2} - \alpha} \
\int_{\frac{1}{n}}^{\frac{\pi}{2}} \
 r^{ \frac{N - 1}{2}} dr \ + \ \frac{1}{n} \ \int_{\frac{1}{n}}^{\frac{\pi}{2}} \
 dr^N \Big{)} \ f^{*}(\overline{x}) \ \ \Big{\}}
$$

Taking into consideration that   $\alpha > \frac{N - 1}{2}$ , \
one can obtain that all integrals in last inequality finite. Then
using definition of maximal function we obtain estimation (2.10).
Theorem  2.2  is proved
\end{proof}

Note that theorem 2.2. specified obtained earlier estimation of
maximal operator.\footnote{Estimation \ (2.10) \ specifies a
similar estimation obtained by A. Bonami , J. Clerc \cite{BonCle1}
and such a specification was required due to necessity to estimate
Chezaro means of critical exponent (see in section 3 below). For
the first time estimation \ (2.10) \ was obtained by the author of
the present paper in \cite{Rakh1} (see also in \cite{Rakh2})} From
the theorem 2.2 it is easy to obtain following corollaries.

\begin{corollary}\label{cor}
Let   $f \in L_1 \big{(} S^N \big{)}$  .  If    $\alpha > \frac{N
- 1}{2}$ , then  for maximal operator (2.6) following estimation
if valid
\begin{gather}\label{rce2}
 mes \big{\{} x \in \big{(} S^N \big{)} : S^{\alpha}_{*} f(x) >
 \mu > 0 \big{\}} \ \leq \ \frac{c \ \|f\|_1}{\mu}
\end{gather}
\end{corollary}
\begin{proof}    Inequality (2.16) immediately follows from
estimation (2.10) and inequality  (2.8). Corollary is proved.
\end{proof}
\begin{corollary}\label{cor} Let  $f \in L_1 \big{(} S^N \big{)}$  and     $\alpha > \frac{N
- 1}{2}$   .  Then almost everywhere on sphere
\begin{gather}\label{rce2}
 \lim_{n \rightarrow \infty} S^{\alpha}_{n} f(x) = f(x).
\end{gather}
 \end{corollary}
\begin{proof} Equality (2.17) follows from estimation (2.16) .
Corollary is proved.
\end{proof}
Using interpolation theorem of E. Stein \cite{SteWei1} and
estimation for Hardy-Littlwood maximal function (2.9)  from the
theorem 2.2 we obtain (see: \cite{BonCle1} )
\begin{corollary}\label{cor}
Let  $f \in L_p \big{(} S^N  \big{)} , \ p > 1 $ \  and  \ $\alpha
> (N - 1) \big{(} \frac{1}{p} - \frac{1}{2} \big{)}$ .
Then
\begin{itemize}
\item[$(i)$]  For maximal operator   (2.6)  following inequality is true
\begin{gather}\label{rce2}
 \| S^{\alpha}_{*} f \|_p \ \leq \ c \| f \|_p
\end{gather}
\item[$(ii)$]  Almost everywhere on sphere we have following equality
\begin{gather}\label{rce2}
 \lim_{n \rightarrow \infty} S^{\alpha}_{n} f(x) = f(x).
\end{gather}
\end{itemize}
\end{corollary}

\vspace{0.2 cm}
\section{Generalized localization of Chezaro means of order  \ $\frac{N - 1}{2}$ \   of
Fourier-Laplace series of functions from  \ $L_1$  .} \vspace{0.2
cm}

\begin{definition} \label{def1.1}   Let \ $V$ \    a subdomain of a sphere, it can also
coincides with  \ $  S^N  $. We say that for Chezaro means \
$S^{\alpha}_{n} f(x)$ \ it is true generalized principles of
localization in a class of functions  \ $ L_p \big{(}S^N \big{)}
$, if for an arbitrary function \ $f$ \ form \ $ L_p \big{(}S^N
\big{)} $ such  that \ $f(x) = 0 $ \ when \ $x \in V$,\ almost
everywhere on \ $V$
$$ \lim_{n \rightarrow \infty} S^{\alpha}_{n} f(x) = 0. $$
\end{definition}
Main result of the present section is a following theorem.

\begin{theorem} \label{Th.sol}
Let   \ $ f \in L_1 \big{(}S^N \big{)} $       and let  \ $f(x) =
0 $ , when \ $x \in V \subset S^N$ . Then almost everywhere on \
$V$ \ Fourier-Laplace series of a function \ $f$ \ convergence to
zero by Chezaro means of the order \ $\alpha = \frac{N - 1}{2}$  .
\end{theorem}

From theorem 3.2. it follows that principles of generalized
localization for the Fourier-Laplace series on a sphere is valid
in critical index \ $\alpha = \frac{N - 1}{2}$  of Chezaro means
in a class of functions   \ $L_1$. In order to prove the theorem
we should estimate maximal operator \ (2.6) \ with critical
exponent.\footnote{Exponent  \quad $\frac{N - 1}{2}$ \quad was
termed by Bochner \cite{Boc} as the critical exponent. It was, in
particular, justified by the fact that localization principal
holds for spherical partial Fourier integrals with the critical
(and above) exponent, and but not for exponent below it. However,
for Fourier multiple series (as well as for Fourier-Laplace series
on sphere) with  spherical sums localization principle (in usual
sense) is not correct in critical exponent \cite{SteWei1}.
Localization of the critical exponent for the Chezaro means of
spherical expansions of distributions are studied in  \cite{Rakh3}
- \cite{Rakh5}. Meantime a localization of the partial sums (means
of zero exponent) in class \quad $L_2$ \quad is another critical
case (Luzin's problem, see \cite{Ali}). This problem was studied
by Bastis I.J. in \cite{Bas} and  Meaney C. in \cite{Mean}.}

\begin{lemma}\label{lemma1}
 Let   \ $f(x)$ \  a summable function on a sphere and let denotes
   \ $\overline{x}$ \  diametrically opposite point to \ $x$
 and let  \ $\alpha \geq \frac{N - 1}{2}$  . If    \ \  $f(x) =
0 , $ \ when $ \ x \in V \subset S^N$,  \ then there exists a
constant   \ $c_{\alpha}$ \   not depending on \ $f$ \
        such that at any point of the domain   \ $V_1 \subset V$   ,
         where \ $\gamma(V_1, \partial V) > 0$  , following inequality is valid
\begin{gather}\label{rce2}
 S^{\alpha}_{*} f(x) \leq c_{\alpha} \ f^{*}(\overline{x}).
\end{gather}
here \ \ $f^{*}$   \ \ is  Hardy-Littlwood's maximal function,  \
$\partial V$ \ is a spherical boundary of the domain \ $V$   .
\end{lemma}
\begin{proof} We will choose a constant  \ $r_0 , \ (r_0 > 0)$  , \  such that for an arbitrary  \ $x$ \
     from \ $ V_1$ \
spherical ball of a radius equal to  \ $r_0 $   \ with the center
at  \ $x$ \ is a subset of  \ $V$  . For estimation of \
$S^{\alpha}_{*} f(x) $ \ we transform integral in right side of
(2.4) as $$ \int_{\gamma \leq r_0}  \ \ + \ \ \int_{r_0 \leq
\gamma \leq \pi - \frac{1}{n}} \ \ + \ \ \int_{\pi - \frac{1}{n}
\leq \gamma \leq \pi} . $$

 Note that because of \ $f(x) = 0$ \ on \ $V_1$ ,
first integral is equal to zero. Thus from (2.11)-(2.12) we obtain

$$ | S_n^{\alpha} f(x) | \ \leq \ c \ \Big{\{} \ n^{\frac{N - 1}{2}
- \alpha}   \int_{r_{0} \leq \gamma \leq  \pi  -  \frac{1}{n}}
\frac{ | f(y) |}{ \Big{(} \sin \gamma \Big{)}^{ \frac{N - 1}{2}}}
d \sigma (y) \ + \ $$

$$ + \ n^{\frac{N - 3}{2} - \alpha} \ \int_{r_0 \leq \gamma \leq \pi - \frac{1}{n}}\ \frac{| f(y) |}
{\Big{(} \sin \gamma \Big{)}^{ \frac{N + 1}{2}}} d \sigma(y) \ + \
$$

\begin{gather}\label{rce2}   + \ n^{-1} \ \int_{r_0 \leq \gamma \leq \pi - \frac{1}{n}}\
  | f(y) | d \sigma(y) +
  \ n^{ N - 1 - \alpha} \ \int_{ \frac{1}{n} \leq \gamma \leq \pi }  \ | f(y)
| d \sigma(y) \Big{\}} \end{gather} where a constant \ $c$ \
depends only on  \ $r_0$   . \ Then each term in right side of
(3.2) estimate by Hardy-Littlwood's maximal function

$$ | S_n^{\alpha} f(x) | \ \leq \ c \ \Big{\{} \ n^{\frac{N - 1}{2}
- \alpha}   \int_{r_{0} \leq \gamma \leq  \pi  -  \frac{1}{n}}
\frac{ | f(y) |}{ \Big{(} \sin \gamma \Big{)}^{ \frac{N - 1}{2}}}
d \sigma (y) \ + \ $$

$$ + \ n^{\frac{N - 3}{2} - \alpha} \ \int_{r_0 \leq \gamma \leq \pi - \frac{1}{n}}\ \frac{| f(y) |}
{\Big{(} \sin \gamma \Big{)}^{ \frac{N + 1}{2}}} d \sigma(y) \ + \
$$

\begin{gather}\label{rce2}   + \ n^{-1} \ \int_{r_0 \leq \gamma \leq \pi - \frac{1}{n}}\
  | f(y) | d \sigma(y) +
  \ n^{ N - 1 - \alpha} \ \int_{ \frac{1}{n} \leq \gamma \leq \pi }  \ | f(y)
| d \sigma(y) \Big{\}} \end{gather}
 Here a symbol  \quad $\overline{\gamma}$    \quad is a  spherical distance between
points \quad $\overline{x}$ \quad and \quad $y$ ,  \quad
$\overline{\gamma} = \gamma (\overline{x}, y)$   . \quad Replacing
with the maximal function we obtain
\begin{gather}\label{rce2} | S_n^{\alpha} f(x) | \ \leq \ c_0 \
\Big{\{}
 \ 1 + \ \int_{\frac{1}{n}}^{ \pi - r_0} \
  \big{(} \sin r \big{)}^{- \frac{N - 1}{2}} \ r^{N - 1 } dr \Big{\}}  f^{*}(\overline{x})  \end{gather}
 and a constant   \quad $c_0$ \quad    in  (3.4) does not have singularities when \quad $\alpha = \frac{N - 1}{2}$  .
Lemma 3.3  is proved. \end{proof}

\begin{lemma}\label{lemma1}
Let \quad $f \in L_1(S^N)$  and this function satisfies the
conditions of lemma 3.3 . \quad If \quad $\alpha = \frac{N -
1}{2}$ , \quad then for maximal operator (2.6) following
estimation is valid
\begin{gather}\label{rce2}
mes \Big{\{} x \in V_1 : S^{\alpha}_*f(x) \ > \ \mu \ > \ 0 \
\Big{\}} \ \leq \ c \frac{\parallel f \parallel_1}{\mu}
\end{gather}

\end{lemma}

\begin{proof} Inequality  (3.5) immediately follows from lemma 3.3  and from the
fact that Hardy-Littlwood's maximal function is weak bounded from
\quad $ L_1(S^N)$    \quad to   \quad $ L_1(S^N)$    \quad . Lemma
3.4   is proved.
\end{proof}

Statements of theorem 3.2 immediate follows from estimation (3.5).


\begin{thebibliography}{99}

\bibitem{Ali}
 Alimov Sh.A., I'lin V.A. and  Nikishin E.M.: Convergence problems of
multiple thrigonomotric series and spectral expansions, {\it
Uspehi Mat Nauk SSSR,} {\bf v.31 no 6}, 145-153 (1976)

\bibitem{Bas} Bastis I.J. , : On almost everywhere convergence eigenfunction expansions of Laplce operator on a sphere, {\it Math Notes,} {\bf v33  no 6}, 857-862 (1983).

\bibitem{Boc} Bochner S.: Summation of multiple Fourier series by spherical means, {\it  Trans. Amer. Math. Soc.}  {\bf 40 } , 175-207,
(1936).

\bibitem{BonCle1}   Bonami A.,  Clerck J. : Sommes de Cesaro et multiplicateurs des
development en harmoniques spheriques, {\it Trans. Amer. Math.
Soc.}, {\bf v.183} ,   223-263, (1973).

\bibitem{Kacz}  Kaczmarz S. and  Steinhaus H.: Theorie der Orthogonalreihen,
Warsaw, 1935 (Chelsea, 1951).

\bibitem{kog1}  Kogbetliantz E. :
{\it J. Math.Pres Appl.} {\bf 9 } , 3 - 107, (1924).

\bibitem{Mean} Meaney C. , : Localization of spherical harmonic expansions,{\it Monatsh. Math. ,} {\bf 98, no. 1} ,  65-74 (1984).

\bibitem{Rakh1} Rakhimov A.A., : On the almost everywhere localization of
Fourier-Laplace series, {\it Manuscript of Diploma Thesis,
Preprint, Tashkent Sate University, } {\bf } 1-14 (1983).

\bibitem{Rakh2} Rakhimov A.A., : On the localization almost
everywhere Chezaro means of Fourier-Laplace series on
N-dimensional sphere, {\it J. Izvestiya of Uzbek Academy of
Science,} {\bf  no 2}28-33 (1987).

\bibitem{Rakh3} Rakhimov A.A., : On the localization of spectral expansions
 of distributions on a compact Reamanian manifold with boundary, {\it Transactions
  of the International Conference on Modern problems of Mathematical physics and information Technologies, } {\bf v.1} 147-150 (2005).

\bibitem{Rakh4} Rakhimov A.A., : On The Localization Conditions of Reiez's Means of Spectral Expansions Connected with
 the Laplace Operator on Sphere, {\it Proceedings of International Conference "Tikhonov and contemporary mathematics", Russia,} {\bf } 216-218 (2005).

\bibitem{Rakh5} Rakhimov A.A., : On the localization conditions of the Reiszean means of distribution expansions
 of the Fourier series by eigenfunctions of the Laplace operator on sphere, {\it J. Izvestiya Vuzov, Phys. Math. Sci., } {\bf v.3-4} 47-50 (2003).

\bibitem{shu1} Shubin
M.A. : Pseudodifferential operators {\it Bull. Math. Biol.}, {\bf
47}, 145-153 (1985)

\bibitem{SteWei1}
 Stein E.M.,   Weiss G.: Introduction to Fourier Analysis on
Euclidean Spaces, {\it Princeton,} (1971)

\bibitem{Sze}  Szego G.: Orthogonal
polynomials, {\it Amer. Math. Soc. Colloq. Publ.,} 1939.

\end{thebibliography}
\end{document}